\documentclass[12pt,oneside]{amsart}
\usepackage[utf8]{inputenc}
\usepackage[english]{babel}
\usepackage[utf8]{inputenc}
\usepackage{amsfonts}
\usepackage{amssymb}
\usepackage{amsthm}
\usepackage{graphicx}
\usepackage{comment}
\usepackage{epstopdf}
\usepackage{pgf}
\usepackage{url}
\usepackage{srcltx}

\newcommand{\cupdot}{\mathbin{\mathaccent\cdot\cup}}

 \theoremstyle{plain}
      \newtheorem{theorem}{Theorem}[section]
      \newtheorem{lemma}{Lemma}[section]
      
      \newtheorem{problem}{Problem}
      \newtheorem{remark}{Remark}[section]
      \theoremstyle{definition}
      \newtheorem{definition}{Definition}[section]

\begin{document}

\title[]{On a functional equation related to a pair of hedgehogs with congruent projections}

\author{Sergii Myroshnychenko}
\address{Department of Mathematical Sciences, Kent State University,
Kent, OH 44242, USA} \email{smyroshn@kent.edu}

\begin{abstract}
Hedgehogs are geometrical objects that describe the Minkowski differences of arbitrary convex bodies in the Euclidean space $\mathbb{E}^n$. We prove that two hedgehogs in $\mathbb{E}^n, n \geq 3$, coincide up to a translation and a reflection in the origin, provided that their projections onto any two-dimensional plane are directly congruent and have no direct rigid motion symmetries. Our result is a consequence of a more general analytic statement about the solutions of a functional equation in which the support functions of hedgehogs are replaced with two arbitrary twice continuously differentiable functions on the unit sphere.
\end{abstract}

     \maketitle

\section{Introduction}
In this paper we address several questions related to the following open problem (cf. \cite{Ga}, Problem 3.2, page 125):

\begin{problem}\label{problem}
Suppose that $2 \leq k \leq n-1$ and that $K$ and $L$ are convex bodies in $\mathbb{E}^n$ such that the projection $K|H$ is directly congruent to $L|H$ for all subspaces $H$ in $\mathbb{E}^n$ of dimension $k$. Is $K$ a translate of $\pm L$?
\end{problem}
\sloppy
Here, we say that two sets $A$ and $ B$ in the Euclidean space $\mathbb{E}^k$ are \emph{directly congruent} if there exists a rotation $\phi \in SO(k)$, such that $\phi(A)$ is a translate of $B$.

We refer the reader to \cite{Go}, \cite{Ga} (pp. $100-110$), \cite{Ha} (pp. $126-127$), \cite{R1}, \cite{R3}, \cite{ACR} for history and partial results related to this problem. In particular, V. P. Golubyatnikov considered Problem \ref{problem} in the case $k=2$ and obtained the following result.

\begin{theorem}[\cite{Go}, Theorem 2.1.1, page 13] \label{th2}
Consider two convex bodies $K$ and $L$ in $\mathbb{E}^n, n \geq 3$. Assume that their projections on any two-dimensional plane passing through the origin are directly congruent and have no direct rigid motion symmetries, then $K = L + b$ or $K = -L +b$ for some $b \in \mathbb{E}^n$.
\end{theorem}

 Here a set $A \subset \mathbb{E}^2$ has a direct rigid motion symmetry if it is directly congruent to itself.

In this paper we study a functional equation related to Problem \ref{problem} in the case $k=2$. To formulate our main result we define an analogue of the notion of a direct rigid motion symmetry for functions on the unit circle $S^1$ in $\mathbb{E}^2$. We say that a function $h$ on $S^1$ satisfies a \emph{direct rigid motion symmetry equation} if there exists a non-trivial rotation $\phi\in SO(2)$ and $a \in \mathbb{E}^2$, such that
\begin{equation}\label{eq1}
h(\phi( u) )+a\cdot u = h(u) \quad \textrm{for any} \quad u \in S^1.
\end{equation}

Our main result is

\begin{theorem}\label{th}
Let $f$ and $ g$ be two twice continuously differentiable real-valued functions on $S^{n-1} \subset \mathbb{E}^n$, $n \geq 3$. Assume that for any $2$-dimensional plane $\alpha$ passing through the origin there exists a vector $a_{\alpha} \in \alpha$ and a rotation $\phi_{\alpha} \in SO(2,\alpha)$, such that the restrictions of $f$ and $g$ onto the large circle $S^{n-1} \cap \alpha$ satisfy the equation
\begin{equation} \label{eq1}
f(\phi_{\alpha}( u) )+a_{\alpha}\cdot u = g(u) \quad \forall u \in S^{n-1} \cap \alpha.
\end{equation}
Then there exists $b \in \mathbb{E}^n$ such that for all $u \in S^{n-1}$ we have $g(u) = f(u) + b \cdot u$ or $g(u) = f(-u) + b \cdot u$, provided that the restrictions of $f, g$ onto any such large circle $S^{n-1} \cap \alpha$ do not satisfy the direct rigid motion symmetry equation.
\end{theorem}

If $f$ and $g$ are the support functions of convex bodies $K$ and $L$ in $\mathbb{E}^n, n \geq 3$, respectively, we reproduce the aforementioned result of \mbox{V. P. Golubyatnikov}, \cite{Go}. Our approach is based on his ideas together with an application of the connection between twice continuously differentiable functions on the unit sphere and support functions of convex bodies. It allows, in particular, to get rid of the convexity assumption on functions.

In the case when the orthogonal transformations $\phi_{\xi}$ degenerate into identity or reflection with respect to the origin, we show that the assumptions on the lack of symmetries and smoothness are not necessary. We have

\begin{theorem}\label{th3}
Let $2 \leq k \leq n-1$ and let $f, g$ be two continuous real-valued functions on $S^{n-1} \subset \mathbb{E}^n$. Assume that  for any $k$-dimensional plane $\alpha$ passing through the origin and some vector $a_{\alpha}\in \alpha$, the restrictions of $f$ and $g$ onto $S^{n-1}\cap \alpha$ satisfy at least one of the equations
$$
f(-u)+a_{\alpha} \cdot u = g(u) \quad \textrm{for all} \quad u \in \alpha \cap S^{n-1},\textrm{or}
 $$
 $$
f(u) + a_{\alpha} \cdot u = g(u) \quad \textrm{for all} \quad u \in \alpha \cap S^{n-1}.
$$
Then there exists $b \in \mathbb{E}^n$ such that for all $u \in S^{n-1}$ we have $g(u) = f(u) + b \cdot u$ or $g(u) = f(-u) + b \cdot u$.
\end{theorem}

As one of the applications of Theorem \ref{th} we also obtain a result about the classical hedgehogs, which are geometrical objects that describe the Minkowski differences of arbitrary convex bodies in $\mathbb{E}^n$.

The idea of using Minkowski differences of convex bodies may be traced back to some papers by \mbox{A. D. Alexandrov} and H. Geppert in the 1930's (see \cite{A}, \cite{Ge}). Many notions from the theory of convex bodies carry over to hedgehogs and quite a number of classical results find their counterparts (see, for instance, \cite{MM4}). Classical hedgehogs are (possibly singular, self-intersecting and non-convex) hypersurfaces that describe differences of convex bodies with twice continuously differentiable support functions in $\mathbb{E}^n$. We refer the reader to works of \mbox{Y. Martinez-Maure}, \cite{MM1}, \cite{MM2}, \cite{MM3}, for more information on this topic.

We have

\begin{theorem}\label{hedg}
Consider two classical hedgehogs $H_f$ and $H_g$ in $\mathbb{E}^n, n \geq 3$. Assume that their projections on any two-dimensional plane passing through the origin are directly congruent and have no direct rigid motion symmetries, then $H_g = H_f + b$ or $H_g = -H_f +b$ for some $b \in \mathbb{E}^n$.
\end{theorem}

It remains unclear if Theorem \ref{th} holds without the assumption that the restrictions of $f$ and $g$ to any equator do not satisfy the direct rigid motion equation.

\section{Notation, Auxiliary Definitions and General Remarks}
We denote by $S^{n-1}$ the set of all unit vectors in the Euclidean space $\mathbb{E}^n$. $SO(n)$ is defined to be the set of all linear orthogonal transformations of $\mathbb{E}^n$ that can be represented as matrices with determinant equal to 1. For any unit vector $\xi \in S^{n-1}$ we denote $\xi^{\perp}$ to be the orthogonal complement of $\xi$ in $\mathbb{E}^{n}$, i.e. the set of all $x \in \mathbb{E}^{n}$ such that $x \cdot \xi = 0$. For any function $h$, $h_e$ and $h_o$ stand for its even and odd parts respectively,
\begin{equation}\label{EvenOdd}
 h_e(u) = \frac{h(u)+h(-u)}{2} \quad \textrm{and} \quad h_0(u) = \frac{h(u)-h(-u)}{2}.
\end{equation}

Observe that functions in equation (\ref{eq1}) can be considered up to translations. Namely, if instead of the function $f(v)$ on $S^{n-1}$ we consider $f_1(v) = f(v) + y \cdot v$ for any $y \in \mathbb{E}^n$, then $f_1(v)$ satisfies equation (\ref{eq1}) with some other vector $b_{\alpha} = a_{\alpha} - \phi^T_{\alpha} (y|_{\alpha}) \in \alpha$,
\begin{equation}\label{shift}
f_1(\phi_{\alpha}(v)) + b_{\alpha} \cdot v = f(\phi_{\alpha}(v)) + y \cdot \phi_{\alpha}(v) + (a_{\alpha} - \phi_{\alpha}^T(y|_{\alpha})) \cdot v = g(v).
\end{equation}

Here, $\phi_{\alpha}^T$ stands for the conjugate of $\phi_{\alpha}$, and $y|_{\alpha}$ is the projection of $y$ on $\alpha$.

In the case $n=3$, for any two-dimensional plane $\alpha$ passing through the origin there exists $\xi \in S^2$, such that $\xi^{\perp} = \alpha$. In this case, we will denote $\phi_{\alpha}$ by $\phi_{\xi}$.
It is well-known that any rotation in $SO(3)$ is determined by an axis of rotation and an angle of rotation (Euler's rotation theorem). Following \cite{Go}, for a fixed orientation in $\mathbb{E}^3$ we consider the map $\Phi: S^2 \to SO(3)$, defined as $\Phi(\xi) = (\xi,\varphi_{\xi})$, i.e. $\Phi(\xi)$ is a rotation around the direction $\xi$ by the angle $\varphi_{\xi}$, whose restriction to $\xi^{\perp}$ coincides with the rotation $\phi_{\xi}$ in equation (\ref{eq1}). Here $\varphi_{\xi}\in [-\pi,\pi]$ is the least angle of rotation (in absolute value), corresponding to the rotation $\phi_{\xi}$; $\phi_{\xi} = \Phi(\xi)|_{\xi^{\perp}}$ and we write $ \phi_{\xi} \in (SO(2), \xi^{\perp})$. We identify the ends of the interval $[-\pi,\pi]$, since the plane rotations by the angle $\pi$ and $-\pi$ coincide. We see that $\varphi_{-\xi} = -\varphi_{\xi}$ and instead of $\Phi$ we will consider the map $\varphi: S^2 \to [-\pi,\pi]$, $\varphi(\xi)=\varphi_{\xi} $.

Also, for any $\beta \in [-\pi,\pi]$ denote by $ \varphi^{-1}(\beta) = \{\xi \in S^2: \varphi_{\xi} =\beta\}$. For convenience, any great circle on $S^2$ orthogonal to $u$ will be denoted by $E(u)= S^2 \cap u^{\perp}$.

Given any twice continuously differentiable real-valued function $h(u)$ on $S^{n-1}$, \emph{the classical hedgehog} $H_h$ with support function $h$ is defined as the envelope $H_h \subset \mathbb{E}^{n}$ of the family of hyperplanes determined by $h(u) = x \cdot u$ for any $x \in \mathbb{E}^{n}$.
A projection of a classical hedgehog $H_h$ onto a subspace $\alpha$ is the envelope of hyperplanes in $\alpha$ defined by $h|_{\alpha}(u) = x \cdot u$ for  $u\in \alpha \cap S^{n-1}$ and $x \in \alpha$, which is also a classical hedgehog $H_{h|_{\alpha}} \subset \alpha$ (see \cite{MM4}).

\section{Proof of the Main Result in the Case $n=3$}

\subsection{Idea of the proof}
 The main idea of the proof of Theorem \ref{th} is to reduce the matter to the case or translations and reflections only, i.e. to show that the rotations $\phi_{\xi}$ in (\ref{eq1}) can only be trivial or by angle $\pi$.

\subsection{Plan of the proof}
Our goal is to show that $\varphi^{-1}(0) = S^2$ or $\varphi^{-1}(\pi) = S^2$.

We start by proving that, without loss of generality, one can assume that the functions in Theorem \ref{th} are odd (see Lemma \ref{even}).

Using our assumption that the restrictions of $f$ and $g$ do not satisfy the direct rigid motion symmetry equation in any  equator, we will show that the map $\varphi$ is continuous (see Lemma \ref{cont}); and that, due to the oddness of $\varphi$, $\varphi_{-\xi} = - \varphi_{\xi}$, one of the sets $\varphi^{-1}(0)$ or $\varphi^{-1}(\pi)$ is not empty. In fact, we show that if $S^2 \neq \varphi^{-1}(0) \cup \varphi^{-1}(\pi)$, then one of the sets $\varphi^{-1}(0)$ or $\varphi^{-1}(\pi)$ intersects all meridians joining $u_0$ and $-u_0$, where $u_0 \in S^2 \setminus (\varphi^{-1}(0) \cup \varphi^{-1}(\pi))$ (see Lemma \ref{homology}).

In order to show that $\varphi^{-1}(0) = S^2$ or $\varphi^{-1}(\pi) = S^2$, we will prove that it is enough to consider two cases: the set $\varphi^{-1}(0)$ (or $\varphi^{-1}(\pi)$) is not a great circle and $\varphi^{-1}(0)$ (or $\varphi^{-1}(\pi)$) is a great circle.

If the set $\varphi^{-1}(0)$  (or $\varphi^{-1}(\pi)$) is not a great circle on $S^2$, our argument is based on the observation that $\varphi^{-1}(0)$ contains three non-coplanar vectors (see Lemma \ref{dense}). This helps to reduce the proof to the case of translations and reflections only (see Lemma \ref{suss2} and the argument after it).

If the set $\varphi^{-1}(0)$  (or $\varphi^{-1}(\pi)$) is a great circle on $S^2$, we use the result from \cite{S} and reduce condition (\ref{eq1}) to a similar equation on support functions of convex bodies of constant width (see (\ref{eq2}) and Lemma \ref{width}).

 We finish the proof in the case $n=3$ by showing that, for convex bodies of constant width, Hadwiger's result \cite{Ha} holds for a circle of directions $\varphi^{-1}(0)$ instead of a cylindrical set of directions.

\subsection{Auxiliary Lemmata}
Our first observation is

\begin{lemma}\label{even}
If $f$ and $g$ verify equation (\ref{eq1}) in the case $n=3$, then $f_e = g_e$ on $S^2$.
\end{lemma}

\begin{proof}
Comparing the even parts of equation (\ref{eq1}) we have:
$$
f_e(\phi_{\xi} (u)) = g_e(u) \quad \textrm{for any} \quad \xi \in S^2 \quad \textrm{and} \quad u \in \xi^{\perp} \cap S^2.
$$

Applying the Funk transform, and using the invariance of the Lebesgue measure under rotations, we obtain:
$$
\int_{E(\xi)} f_e(\phi_{\xi} (u)) d\sigma(u) = \int_{E(\xi)} f_e(u) d\sigma(u)= \int_{E(\xi)} g_e(u) d\sigma(u).
$$
Since the Funk transform is injective on even functions (see \cite{He}, Corollary 2.7, p. 128), we obtain the desired result.
\end{proof}

By the previous lemma, from now on we may assume that our functions $f$ and $g$ are odd.

If $\varphi \equiv 0$, Theorem \ref{th}  follows from

\begin{lemma}[cf. \cite{R3}] \label{suss}
Let $f, g$ be two continuous functions on $S^{n-1}$ such that for any $\xi \in S^{n-1}$ there exists $a_{\xi} \in \xi^{\perp}$ and
$$
f(u) + a_{\xi} \cdot u = g(u) \quad \textrm{for any} \quad u \in E(\xi).
$$
Then there exists $b \in \mathbb{E}^n$, such that $g(u) = f(u) + b \cdot u$ for any $u \in S^{n-1}$.
\end{lemma}

\begin{proof}
For any $u \in S^{n-1}$, consider $F(u)=g(u) - f(u)$ and extend it to $\mathbb{E}^n$ by homogeneity of degree $1$. Then $F(u)$ is continuous on $S^{n-1}$ and for a fixed $\xi \in S^{n-1}$ and any $x \in \xi^{\perp}$ we have $F(x) = a_{\xi} \cdot x$.

We claim that $F$ is linear in $\mathbb{E}^n$. Choose any $v_1, v_2 \in S^{n-1}$ and $c_1,c_2 \in \mathbb{R}$. Then,
$$
F(c_1v_1 + c_2v_2) = a_{\xi} \cdot (c_1v_1 + c_2v_2)
$$
for $\xi \perp span\{v_1,v_2\}$. On the other hand, we have
$$
c_1F(v_1)+c_2F(v_2) = c_1\left(a_{\xi} \cdot v_1\right)+ c_2\left(a_{\xi} \cdot v_2\right),
$$
since $v_1 \perp \xi$ and $v_2 \perp \xi$.
The linearity $F(c_1 v_1 + c_2 v_2) = c_1 F(v_1) + c_2 F(v_2)$ follows.
\end{proof}

\begin{remark}
If $\varphi \equiv \pi$, we may consider the function $f(-u)$ instead of $f(u)$ to conclude that $g(u) = f(-u)+b\cdot u$.
\end{remark}

\begin{lemma}\label{cont}
Let $f, g$ be two continuous functions on $S^2$. If the restrictions of $f, g$ do not satisfy the direct rigid motion symmetry equation in any equator, then the map $\varphi(\xi) = \varphi_{\xi}, \varphi: S^2 \to [-\pi,\pi]$ is continuous on $S^2$
\end{lemma}

\begin{proof}
Let $w_0$ be any point on $S^2$. Consider a sequence of points $\{w_m\}_{m \in \mathbb{N}}$ on $S^2$ such that $\lim_{n\to \infty}w_n = w_0$, and assume that $ \lim_{n\to \infty} \varphi_{w_n} \neq \varphi_{w_0}$. Since $S^2$ is a compact set, there exists a subsequence $\{w_{m_l}\}$, for which $\lim_{l \to \infty} \varphi_{w_{m_l}}= \varphi_1 \neq \varphi_{w_0}$. This implies that
$$
f(\varphi_{w_0}(u)) + b_{w_0} \cdot u = g(u) \quad \textrm{for any} \quad u \in w_0^{\perp} \quad \textrm{and some} \quad b_{w_0} \in w_0^{\perp},
$$
and
$$
f(\varphi_1(u)) + a_{w_0} \cdot u = g(u) \quad \textrm{for any} \quad u \in w_0^{\perp}  \quad \textrm{and some} \quad a_{w_0} \in w_0^{\perp}.
$$

Combining the above two equations, we obtain
$$
f(\varphi_{w_0}(u)) + (b_{w_0}-a_{w_0}) \cdot u= f(\varphi_1(u)) \quad \textrm{for any} \quad u \in w_0^{\perp}.
$$

The last equation is the equation of the direct rigid motion symmetry for the function $f$, that cannot be satisfied by the condition of the lemma. Thus, we obtain a contradiction to the assumption of discontinuity of the map $\varphi$.
\end{proof}

Now assume
\begin{equation}\label{contradiction}
\exists u_0 \in S^2: 0 < \varphi(u_0) < \pi.
\end{equation}
Consider the set of all meridians $M(u_0)=\{m_t:t \in [0,2\pi]\}$ that connect $u_0$ and $-u_0$. Each meridian $m_t$ corresponds to a unique point of intersection with $E(u_0)$ and the great circle $E(u_0)$ can be parameterized by the natural parameter $t$.


Our next Lemma is Lemma 2.1.2, from (\cite{Go}, p. 15). We give a more detailed proof for the convenience of the reader.

\begin{lemma}\label{homology}
Let $\varphi$ be continuous on $S^2$ and assume (\ref{contradiction}) holds. Then one of the sets $\varphi^{-1}(0)$ or $\varphi^{-1}(\pi)$ intersects all the meridians in $M(u_0)$.
\end{lemma}

\begin{proof}
Parameterize each meridian $m_t = m_t(s)$ by a natural parameter $s \in [-\frac{\pi}{2},\frac{\pi}{2}]$, such that for any $t$ we have $u_0 = m_t(\frac{\pi}{2})$ and $-u_0 = m_t(-\frac{\pi}{2})$. By the continuity of $\varphi$, we see that the restriction $\varphi|_{m_0}$ of $\varphi$ to the meridian $m_0$ satisfies
\begin{equation}\label{hom}
Im(\varphi|_{m_0}) \cap \{\pi\} \neq \emptyset \quad \textrm{or} \quad Im(\varphi|_{m_0}) \cap \{0\} \neq \emptyset.
\end{equation}
Similarly, one can obtain the analogue of the above for any $m_t$, $t \in [0, 2\pi]$. The idea of the proof is to use the fact that homotopy equivalent spaces (meridians) have isomorphic homology groups (see \cite{Hat}, p. 111). Let $\tilde{m}_0$ be the meridian $m_0$ with its poles identified, so that it becomes $S^1$ (see Figure \ref{8}). Consider the map
$$
\mu_0 = \varphi|_{\tilde{m}_0}: S^1 \to \textrm{``figure eight"} \infty.
$$

Here, $\mu_0$ maps the meridian $\tilde{m}_0$ with the identified poles $u_0$ and $-u_0$ into the interval $[-\pi,\pi]$, where the pair of points $-\varphi(u_0)$ and $\varphi(u_0)$ and also $-\pi$ and $\pi$ are identified respectively, so that it looks like $\infty$.

\begin{figure}
  \centering
  \includegraphics[width=12cm]{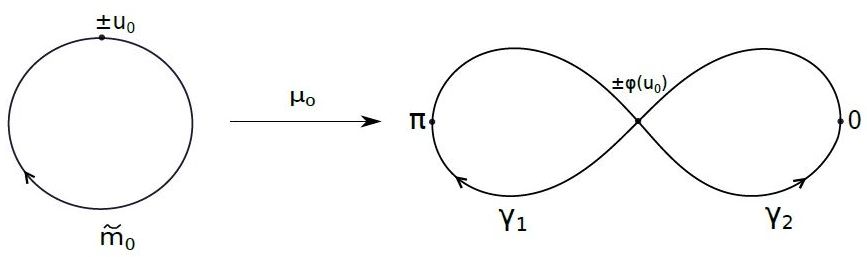}\\
  \caption{Identifying poles of the meridian $m_0$}\label{8}
\end{figure}

It is known (see \cite{Hat}, p. 106, and Exercise 31, p. 158) that the first homology groups $H_1$ of the spaces $S^1$ and $\infty$ are
$$
H_1(S^1) \cong \mathbb{Z} \quad \textrm{and} \quad H_1(\infty) \cong \mathbb{Z} \oplus \mathbb{Z}.
$$

For the induced homomorphism $hom_{\mu_0}$ (see \cite{Hat}, p.111) corresponding to $\mu_0$,
$$
hom_{\mu_0}: \mathbb{Z} \rightarrow \mathbb{Z} \oplus \mathbb{Z},
$$
consider the image of $1 \in \mathbb{Z}$, $hom_{\mu_0}(1) = (n_1, n_2) \in \mathbb{Z} \oplus \mathbb{Z}$. Here, $n_1$ corresponds to number of times we loop around the left circle of $\infty$ (on the picture loop $\gamma_1$ going clockwise) and $n_2$ corresponds to the the number of times we loop around the right one (on the picture loop $\gamma_2$ going counterclockwise). The element $1 \in \mathbb{\mathbb{Z}}$ can be thought of as a continuous loop on $S^1$ with the beginning at $-u_0$ and the end at the point $u_0$, where these two points are identified.

For each meridian $m_t$ we similarly identify the poles $-u_0$ and $u_0$ to obtain $\tilde{m}_t$, $t \in [0, 2 \pi]$, $s \in [-\frac{\pi}{2},\frac{\pi}{2}]$. We consider a continuous homotopy $T(t,s) = \tilde{m}_t(s)$. The homotopy of meridians defines the homotopy $\tilde{\mu}_t$ of the mapping $\mu_0$ as the restriction $\varphi|_{\tilde{m}_t}$,
$$
\tilde{\mu}_t = \varphi|_{\tilde{m}_t}: S^1 \to \textrm{''figure eight"} \infty, \quad \textrm{such that} \quad \mu_t(\pm u_0) = \pm \varphi(u_0).
$$
  By \cite{Hat} (p. 111, Proposition 2.9) we have, $hom_{\mu_0} = hom_{\mu_t}$, and we conclude that $(n_1,n_2)$ does not depend on $t$.

Now we claim that the number $n_1 + n_2$ is odd. If we start changing the parameter $s$ on $\tilde{m}_t(s)$ continuously (from $-\pi/2$ to $\pi/2$) then the image of the map $\tilde{\mu}_t$ is a continuous path on $\infty$ with the beginning at $-\varphi(u_0)$ and the end at $\varphi(u_0)$ (which are identified). This path loops around each side of $\infty$ a number of times. Looping once around either side is equivalent to having a path starting at $-\varphi(u_0)$ and ending at $\varphi(u_0)$, looping twice is equivalent to having a path starting at $-\varphi(u_0)$ and ending at the same point. The same idea can be extended to any even or odd number of loops: if we loop around either side of $\infty$ an odd number of times we start at the point $-\varphi(u_0)$ and end at the point $\varphi(u_0)$; if we loop around either side of $\infty$ an even number of times we start and end at the same point $-\varphi(u_0)$. By adding the number of loops around each side we see that the number of loops $n_1+n_2$ must be odd.

Since $n_1+n_2$ is odd for any $t$, either $n_1$ or $n_2$ is odd. We conclude that we loop around at least one side of $\infty$. Indeed, assume that $\{\pi\}$ is not in the image of $\tilde{\mu}_t$. Then we do not loop around the left circle of $\infty$ at all, in which case $n_1 = 0$. If, on the other hand, $\{0\}$ is not in the image of $\tilde{\mu}_t$, then we do not loop around the right circle of $\infty$ at all, in which case $n_2 = 0$. Thus, either $\{0\}$ or $\{\pi\}$ is in the image of $\tilde{\mu}_t$ for any $t \in [0,2 \pi]$.
\end{proof}

\begin{remark}
Since we can reflect the function $f$ by considering $f(-u)$ instead of $f(u)$, $u \in S^2$, from now on we consider the case when $\varphi^{-1}(0)$ intersects all the meridians in $M(u_0)$.
\end{remark}

The next observation is needed in Lemma \ref{dense}

\begin{lemma}[cf. \cite{Go}, p. 16, Corollary 2.1.1]\label{hom}
If every meridian from $M(u_0)$ intersects $\varphi^{-1}(0)$ at a single point, then $\varphi^{-1}(0)$ is homeomorphic to a circle.
\end{lemma}

\begin{proof}
Since $\varphi$ is continuous, the set $\varphi^{-1}(0)$ is closed.

Consider the map $\Pi: E(u_0) \to \varphi^{-1}(0)$ that maps any point $x \in E(u_0)$ to the point $\Pi(x) \in \varphi^{-1}(0)$, such that $x$ and $\Pi(x)$ belong to the same meridian from $M(u_0)$. The map $\Pi$ is well-defined according to the statement of the lemma. For any point $x \in E(u_0)$, consider a sequence $\{x_k\}_{k \in \mathbb{N}} \subset E(u_0)$, such that $\lim_{k \to \infty} x_k = x$.

Now consider the sequence $\{\Pi(x_k)\}_{k \in \mathbb{N}}$. If $\lim_{k \to \infty} \Pi(x_k)$ does not exist, then $\limsup_{k \to \infty} \Pi(x_k) \neq \liminf_{k \to \infty} \Pi(x_k)$. Then there exists a meridian that contains the two distinct points $\liminf_{k \to \infty} \Pi(x_k)$ and $\limsup_{k \to \infty} \Pi(x_k)$ of the set $\varphi^{-1}(0)$, which contradicts the assumption of the lemma.

If $\lim_{k \to \infty} \Pi(x_k) = z$, the point $z$ belongs to the set $\varphi^{-1}(0)$, since $\varphi^{-1}(0)$ is closed.  If $z \neq \Pi(x)$, using the same argument as above we obtain a contradiction. Thus $\Pi(x) = z$ and $\Pi$ is continuous.

Observe that in the above argument we may interchange the sets $\varphi^{-1}(0)$ and $E(u_0)$ to obtain the continuity of the map $\Pi^{-1}$. Thus $\varphi^{-1}(0)$ is homeomorphic to circle $E(u_0)$.
\end{proof}

\subsection{Case 1: $\varphi^{-1}(0)$ is not a great circle on $S^2$}

\begin{lemma}[cf. \cite{Go}, p. 16, Lemma 2.1.3]\label{dense}
If $\varphi^{-1}(0)$ is not a great circle on $S^2$, then there exist two non-parallel vectors $u_1, u_2 \in \varphi^{-1}(0)$, such that for a dense set of parameters $t\in [0,2\pi]$ the corresponding meridians $m(t)$ in $M(u_0)$ intersect the set $\varphi^{-1}(0)$ at points that are not coplanar with $u_1$ and $u_2$.
\end{lemma}

\begin{proof}
If there exists a meridian $m_{t_1}$, such that the number of points of intersection in $\varphi^{-1}(0) \cap m_{t_1}$ is greater than one, then we take any two points in this intersection to be the required $u_1$ and $u_2$. Any other meridian, except $m_{-\pi + t_1}$, intersects $\varphi^{-1}(0)$ at points that are not coplanar with the above two points.

On the other hand, if every meridian intersects $\varphi^{-1}(0)$ at a single point, then by Lemma \ref{hom} there exists a homeomorphism $\Pi:E(u_0) \to \varphi^{-1}(0)$. By the assumption, $\varphi^{-1}(0)$ is not contained in any great circle on $S^2$. For any $u_1 \in \varphi^{-1}(0)$ fixed any great circle $E_0$ passing through $u_1$ and $-u_1$ (see Figure \ref{s2}). Then consider the family $\{E_s\}_{s \in [0,2\pi]}$ of all great circles passing through $u_1$ parameterized by the angle $s$ that corresponds to the intersection of $E_s$ and $E_0$.

\begin{figure}
  \centering
  \includegraphics[width=9cm]{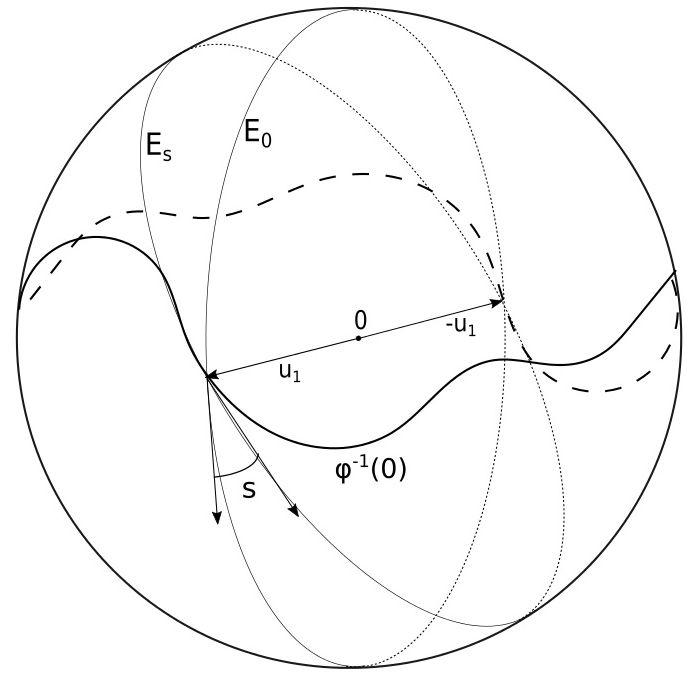}\\
  \caption{The set of directions $\varphi^{-1}(0)$}\label{s2}
\end{figure}

Let $K_s$ denote the set $ E_s \cap \varphi^{-1}(0)$. Then $K_{s_1} \cap K_{s_2} = \{u_1,-u_1\}$ if $s_1 \neq s_2$. This implies that
$$\varphi^{-1}(0) = \cupdot_{s \in [0,2\pi]} (K_s \setminus (\{u_1\} \cupdot \{-u_1\})) \cupdot \{u_1\} \cupdot \{-u_1\},
$$
where $\cupdot$ stands for a disjoint union. Setting $\Pi^{-1}(K_s \setminus (\{u_1\} \cup \{-u_1\})) = G_s \subset E(u_0)$, we have that $E(u_0) = \Pi^{-1}(\varphi^{-1}(0)) =  \cupdot_{s \in [0,2\pi]}G_s \cupdot \{\Pi^{-1}(u_1)\} \cupdot \{\Pi^{-1}(-u_1)\}$. We claim that there exists $G_{s_0}$, such that $G_{s_0}^c$ is dense, or equivalently, $int G_{s_0} = \emptyset$. Assume not, i.e. for any $s \in [0,2\pi]$ we have $int G_{s} \neq \emptyset$, then  $G_s$ contains an open interval of $E(u_0)$, and hence it contains a point that corresponds to a rational value of the parameter $t$ on $E(u_0)$. Thus we obtain a contradiction, since the number of such values of $t$  is countable, but $s$ does not belong to a countable set.

Then the set $B=\Pi(G^c_{s_0})$ is the desired set and it is dense in $\varphi^{-1}(0)$, since homeomorphisms preserve the property of density. We may take $u_2 \in B^c$. By the above, $u_2 \neq \pm u_1$.

\end{proof}

The following lemma is a functional analogue of the result from (\cite{Go}, p. 9, Lemma 1.2.2.)

\begin{lemma}\label{suss2}
Let $f, g$ be two continuous functions on $\mathbb{E}^n$, $n \geq 3$, and let $w_1,w_2,w_3 \subset S^{n-1}$ be non-coplanar vectors. If for any $u_i \in w_i^{\perp},i=1,2,3$ we have
$$
g(u_1) = f(u_1), \quad g(u_2) = f(u_2) \quad \textrm{and} \quad g(u_3) = f(u_3) + a \cdot u_3,
$$
for some $a \in w_3^{\perp}$, then $a = 0$.
\end{lemma}
\begin{proof}
For $i,j =1, 2, 3$, let $P_{i,j} = w_i^{\perp} \cap w_j^{\perp}$. As in the proof of Lemma \ref{suss}, consider the function $F(u) = g(u) - f(u)$. For any $u_{1,3} \in P_{1,3}$ and  $u_{2,3} \in P_{2,3}$ we have
$$
0=F(u_{1,3}) = a \cdot u_{1,3} \quad \textrm{and} \quad 0=F(u_{2,3}) = a \cdot u_{2,3}.
$$
We conclude that $a \perp P_{1,3}$ and $a \perp P_{2,3}$, and so $a \perp span\{P_{1,3},P_{2,3}\}$. On the other hand, $span\{P_{1,3},P_{2,3}\} = w_3^{\perp}$ (this is due to the fact that $\dim P_{1,3} = \dim P_{2,3} = n-2$, $P_{i,3} \subset w_3^{\perp}$ and $P_{1,3} \neq P_{2,3}$) and $a \in w_3^{\perp}$. Thus, $a =0$.
\end{proof}

We will need the following

\begin{lemma}\label{shiza}
Let $f, g$ be two continuous functions on $\mathbb{E}^n, n \geq 3$ and let $w_1, w_2 \in S^{n-1}$. If for some $a_i \in w_i^{\perp}, i=1,2$, we have
$$
f(u_i) + a_i \cdot u_i = g(u_i) \quad \forall u_i \in w_i^{\perp},
$$
then there exists $y \in \mathbb{E}^n$, such that $y \cdot u_i = a_i \cdot u_i$ for any $u_i \in w_i^{\perp}, i=1,2.$
\end{lemma}

\begin{proof}
For any vector $v \in w_1^{\perp} \cap w_2^{\perp}$ we have
$$
f(v) + a_1\cdot v = g(v) = f(v) + a_2 \cdot v.
$$

The above implies that $(a_1 - a_2) \cdot v = 0$. Since the vector $v$ was chosen arbitrary, we have $a_1 - a_2 \perp w_1^{\perp} \cap w_2^{\perp}$, i.e. $a_1 - a_2 = t w_1 + s w_2$ for some $t,s \in \mathbb{R}$. Then the vector $y = a_1 - t w_1 =a_2 + sw_2 $ is the one we need.

\end{proof}

Finally, to obtain a contradiction to our assumption $0< \varphi^{-1}(0) < \pi$, consider two non-parallel vectors $u_1, u_2 \in \varphi^{-1}(0)$ and a point $x$, which belongs to the dense subset $\Pi^{-1}(B) \subset E(u_0)$ defined in Lemma \ref{dense}, so that $z=\Pi(x) \in \varphi^{-1}(0)$ is not coplanar with $u_1$ and $u_2$. Define a function $G(v)$ on $S^2$ to be $G(v) =f(v) + y \cdot v$, where $y$ is the vector obtained by applying Lemma \ref{shiza} to the vectors $u_1$ and $u_2$. Then for any $v_1 \in u_1^{\perp}$
$$
G(v_1)=f(v_1) + a_{u_1} \cdot v_1 = g(v_1)
$$
and for any $v_2 \in u_2^{\perp}$
$$
G(v_2) =f(v_2) + a_{u_2} \cdot v_2 = g(v_2).
$$

Recall that, by the argument in (\ref{shift}), the functions $f$ and $g$ satisfy the equation ($\ref{eq1}$) up to a translation. Hence, for the function $G(u)$, there also exists $a_z\in z^{\perp}$, such that $G(u) + a_z \cdot u = g(u)$ for any $u \in z^{\perp}$. Using Lemma \ref{suss2}, we see that $ a_z = 0$ for any $u \in z^{\perp}$. This implies that $G(u) = g(u)$ for any $u \in E(z).$

Notice also that the set $E(u_0) \cap \{E(z)\}_{z \in B}$ is dense in $E(u_0)$, since $B \subset \varphi^{-1}(0)$ is dense by Lemma \ref{dense}. Both $G$ and $g$ are continuous functions on the sphere, this implies that $G(u) = f(u) + y\cdot u= g(u)$ for any $u \in E(u_0)$. Thus, $\varphi(u_0) = 0$, since otherwise function $f$ would satisfy a direct rigid motion symmetry equation in $E(u_0)$. However, the previous contradicts the assumption $0<\varphi(u_0) < \pi$. We have thus proven Theorem \ref{th} under the hypothesis that $\varphi^{-1}(0)$ is not a great circle.

\subsection{Case 2: $\varphi^{-1}(0)$ is a great circle on $S^2$ }
We use a geometrical approach.

\begin{definition}[\cite{S}, p. 37]
The \emph{support function} $h_V(x)$ of a convex subset $V$ of $\mathbb{E}^n$ is defined as $h_V(x) = \sup \{ x \cdot v: v \in V\}$ for $x \in \mathbb{E}^n$.
\end{definition}

We are going to use the next result to finish the proof of Theorem \ref{th}.

\begin{theorem} [\cite{S}, p.45] \label{shn}
If $f:S^{n-1} \to \mathbb{R}$ is a twice continuously differentiable function, there exists a convex body $K$ and a number $r\geq0$ such that
$$
f(u) + r = h_K(u).
$$
\end{theorem}

Following the proof of Theorem \ref{shn}, one can conclude that the result holds for any larger constant $C \geq r$. Then we may add such a large constant $C$ to both sides of (\ref{eq1}) and extend the functions to $\mathbb{E}^3$ by homogeneity of degree $1$ to obtain another equation
\begin{equation} \label{eq2}
\tilde{f}(\varphi_{\xi}(x))+a_{\xi} \cdot x = \tilde{g}(x) \quad \textrm{for all} \quad x\in \xi^{\perp}, \xi \in S^2,
\end{equation}
where $\tilde{f}(x) = C|x|+f(x)>0$ and $\tilde{g}(x) =C|x| + g(x)>0$ are the support functions of some convex bodies $K_1$ and $K_2$ respectively.

\begin{lemma}\label{width}
The bodies $K_1$ and $K_2$ have the same constant width in any direction.
\end{lemma}

\begin{proof}
Recall that, after Lemma \ref{even}, we assumed that $f$ and $g$ are both odd functions. Let $w_1(u) = \tilde{f}(u)+\tilde{f}(-u)$ be the width of body $K_1$ in the direction $u$. Then
$$
w_1(u) = C+f(u)+C+f(-u) = 2C.
$$
The same can be done for $K_2$ and function $g$.
\end{proof}

It is well-known (see \cite{Ha}) that two convex bodies are translates of each other, provided that their projections in a cylindrical set of directions are translates of each other. Here, a cylindrical set of directions is the set $E(u_0) \bigcup \{ u_0 \}\subset S^2$, for some $u_0 \in S^2$.

The last part of the proof is based on the observation that for two convex bodies of constant width it is enough to consider only a circle of directions $E(u_0)$. This is due to the fact that we can translate the bodies so that their diameters parallel to $u_0$ coincide.

Without loss of generality, we assume now that $\varphi^{-1}(0) = E(e_3)$. Consider two support planes $P_1$ and $P_2$ of $K_2$, which are parallel to $e_3^{\perp}$. Since $K_2$ has constant width, the points $x_1 = K_2 \cap P_1$ and $x_2 = K_2 \cap P_2$ belong to the common perpendicular to these planes (see \cite{Ga}, Lemma 7.1.13, p. 275), which implies that $K_1$ has a parallel translate $K'_1$  tangent to the planes $P_1$ and $P_2$ at the points $x_1$ and $x_2$ respectively. The projections of $K'_1$ and $K_2$ in the directions of the vectors from $e_3^{\perp}$ coincide, and since $\varphi^{-1}(0)$ intersects all the great circles on $S^2$, we obtain that $K'_1 = K_2$. Observe that any shift of any projection would change the values of the support functions on $e_3^{\perp}$ (otherwise the shift would be in the direction orthogonal to $e_3^{\perp}$, which is impossible since the points $x_1$ and $x_2$ are fixed). Similarly, we obtain that $-K_1' = K_2$ in the case $\varphi^{-1}(\pi)=E(e_3) $.

It is known (see \cite{S}, p. 38, Theorem 1.7.1) that if $h_{V_1}, h_{V_2}$ are the support functions of convex subsets $V_1, V_2 \subset \mathbb{E}^n$ respectively and $h_{V_1}(u) =  h_{V_2}(u)$ for any $u \in S^{n-1}$, then $V_1 = V_2$. Thus, we conclude that $\tilde{f}(u) + b \cdot u = \tilde{g}(u)$ in the case $K_1' = K_1 + b$; or $\tilde{f}(-u) + b \cdot u = \tilde{g}(u)$ in the case $-K_1' = -K_1 + b$. Subtracting the constant $C$ from both sides of both equations we conclude that $f(u) + b \cdot u = g(u)$ or $f(-u) + b \cdot u = g(u)$. This finishes the proof of Theorem \ref{th} in the case $n=3$.

\section{Proof of the main result in the case $n >3$}

Theorem \ref{th} in the case $n >3$ is a consequence of Theorem \ref{th} for $n=3$ and Theorem \ref{th3}.

\subsection{Proof of Theorem \ref{th3}}

By induction, it is enough to consider the case $k=n-1, n \geq 3$.

Consider two subsets of $S^{n-1}$,
$$\Xi_0 = \{ \xi \in S^{n-1}: f(v) + a_{\xi} \cdot v = g(v) \quad \forall v \in \xi^{\perp} \quad \textrm{and some}  \quad a_{\xi} \in \xi^{\perp} \} ,
$$
$$
\Xi_{\pi} = \{ \xi \in S^{n-1}: f(-v) + c_{\xi} \cdot v = g(v) \quad \forall v \in \xi^{\perp} \quad \textrm{ and some} \quad c_{\xi} \in \xi^{\perp} \}.
$$
\begin{lemma}\label{cl}
The sets $\Xi_0$ and $\Xi_{\pi}$ are closed.
\end{lemma}
\begin{proof}
Following the argument from \cite{R2} (p.3433, Lemma 5) we may show that the sets $\Xi_0$ and $\Xi_{\pi}$ are closed (for the convenience of the reader we briefly repeat the proof).

Consider a convergent sequence $\{\xi_n\}_{n \in \mathbb{N}} \subset \Xi_0$, $\xi_n \to \xi$. For any $v \in \xi^{\perp}$ we can find a sequence  $\{v_n\}_{n \in \mathbb{N}}, v_n \in \xi_n^{\perp}$ such that $v_n \to v$ as $ n \to \infty$. For these sequences we have $f(v_n) + a_{\xi_n} \cdot v_n = g(v_n)$ and, by compactness, we may assume that $a_{\xi_n} \to b_{\xi} \in \xi^{\perp}$. Then $a_{\xi}\cdot v= g(v) - f(v) = b_{\xi} \cdot v$ for any $v \in \xi^{\perp}$, which implies $a_{\xi} = b_{\xi}$, $\xi \in \Xi_0$. A similar argument can be repeated for $\Xi_{\pi}$ to conclude that both sets $\Xi_0$ and $\Xi_{\pi}$ are closed.
\end{proof}

We will also use
\begin{lemma}[\cite{R2}, p.3431, Lemma 1]\label{rubik}
Let $n \geq 3$, let $f$ and $g$ be two continuous functions on $S^{n-1}$ and let
$$
\Lambda_0 = \{\xi \in S^{n-1}: f(v) = g(v) \quad \forall v \in \xi^{\perp}\},
$$
$$
\Lambda_{\pi} = \{\xi \in S^{n-1}: f(-v) = g(v) \quad \forall v \in \xi^{\perp}\}.
$$
If $S^{n-1} = \Lambda_0 \cup \Lambda_{\pi}$, then $S^{n-1} = \Lambda_0$ or $S^{n-1} = \Lambda_{\pi}$.
\end{lemma}

\begin{lemma}\label{last}
Let $f$ and $g$ be two continuous real-valued functions on $S^{n-1}$, such that $S^{n-1} = \Xi_0 \cup \Xi_{\pi}$, $n \geq 3$. Then $S^{n-1} = \Xi_0$ or $S^{n-1} = \Xi_{\pi}$.
\end{lemma}

We will reduce this lemma to Lemma \ref{rubik}.

\begin{proof}

Since $\Xi_0 \cup \Xi_{\pi} = S^{n-1}$ and since the scalar product $v \to a_{\xi} \cdot v$ is an odd function on $\xi^{\perp}$, we have $f_e \equiv g_e$.

We can assume that $  int (\Xi_0) \neq \emptyset$ and $  int (\Xi_{\pi}) \neq \emptyset$. Indeed, if $  int (\Xi_0) = \emptyset$, then for any $x \in \Xi_0$ there exists a sequence $\{x_n\}_{n \in \mathbb{N}} \subset \Xi_{\pi}$, such that $x_n \to x$. Since $\Xi_{\pi}$ is closed, we obtain $x \in \Xi_{\pi}$ and hence $\Xi_0 \subset \Xi_{\pi}$.

The above implies that there exist two non-parallel vectors $u_1, u_2 \in int (\Xi_0)$. There also exists $w \in int (\Xi_0)$, such that $w$ is non-coplanar with $u_1, u_2$ (otherwise $int (\Xi_0) \subset S_{u_1,u_2}$, where $S_{u_1,u_2}$ is a great circle on $S^{n-1}$ which is spanned by $u_1$ and $u_2$, but that would imply $int (\Xi_0) = \emptyset$). A similar argument can be used to show that there exist three non-coplanar vectors in $int(\Xi_{\pi})$.

By Lemma \ref{shiza}, we may consider a vector $b \in \mathbb{E}^n$, such that the function $F(u) = f_0(u) + u \cdot b$ satisfies $F(v) = g_0(v)$ for any $v \in u_1^{\perp} \cup u_2^{\perp}$. By Lemma \ref{suss2}, this implies that $F(v) = g_0(v)$ for any $ v \in w^{\perp}$ , where $w \in \Xi_0  \backslash S_{u_1,u_2}$. Since $F$ is continuous on $S^{n-1}$, we have $F(v) = f_0(v)+b \cdot v =  g_0(v)$ for any $v \in w^{\perp}$, where $w \in \Xi_0$ and $b$ is independent of $w$. This is due to the fact that $S_{u_1,u_2}$ is nowhere dense in $S^{n-1}$. Similarly, we can show that there exists a vector $c \in \mathbb{E}^n$, such that $f_0(-v) + c \cdot v = g_0(v)$ for any $v \in w^{\perp}$, where $w \in \Xi_{\pi}$ and $c$ is independent of $w$.

The intersection $\Xi_0 \cap \Xi_{\pi} \neq \emptyset$, since $S^{n-1}$ is connected. Consider any $\xi \in \Xi_0 \cap \Xi_{\pi}$ and any $v \in \xi^{\perp}$. Then, $g_0(v) = f_0(v) + b \cdot v = f_0(-v) + c \cdot v$ or
$$
f_0(v) = \frac{c - b}{2} \cdot v, \quad g_0(v) = \frac{c + b}{2} \cdot v \quad \forall v \in \xi^{\perp}, \quad \forall \xi \in \Xi_0 \cap \Xi_{\pi}.
$$

Let $\tilde{f}_0(v) = f_0(v) + y\cdot v$ defined on $S^{n-1}$ for any $y \in \mathbb{E}^n$ . Observe that the set $\Xi_0$ for $f_0$ coincides with the set $\Xi_0$ for the function $\tilde{f}_0$. This is due to the fact that for any $\xi \in \Xi_0$ and $b_{\xi} = a_{\xi} - y|_{\xi^{\perp}} \in \xi^{\perp}$ we have
\begin{equation}\label{shift2}
\tilde{f}_0(v) + b_{\xi} \cdot v = f_0(v) + y \cdot v + ( a_{\xi} - (y|_{\xi^{\perp}}) ) \cdot v= g_0(v) \quad \forall v \in \xi^{\perp}.
\end{equation}
A similar observation holds for  $\xi \in \Xi_{\pi}$ if we put $b_{\xi} = c_{\xi}+y|_{\xi^{\perp}}$. And also, the same holds true for $g_0$, as both functions are interchangable. Hence, by taking $\tilde{f}_0(v) = f_0(v) + \frac{b-c}{2} \cdot v$ and $\tilde{g}_0(v) = g_0(v) - \frac{b+c}{2} \cdot v$, we have $\tilde{f}_0(v) = \tilde{g}_0(v) = 0$ for $\forall v \in \xi^{\perp}, \forall \xi \in \Xi_0 \cap \Xi_{\pi}$. Also, $\tilde{f}_0(v) = \tilde{g}_0(v)$ for any $v \in w^{\perp}, w \in \Xi_0$ and $\tilde{f}_0(-v) = \tilde{g}_0(v)$ for any $v \in w^{\perp}, w \in \Xi_{\pi}$.

Applying Lemma \ref{rubik} to the functions $\tilde{f}_0$ and $\tilde{g}_0$ and the sets $\Lambda_0 = \Xi_0$ and $\Lambda_{\pi} = \Xi_{\pi}$ respectively, we finish the proof of Lemma \ref{last}.
\end{proof}

The proof of Theorem \ref{th3} now follows from the above lemma by induction on $k$.

\subsection{Proof of Theorem \ref{th} for $n >3$}

The proof of Theorem \ref{th} in this case is a direct consequence of Theorem \ref{th3} and the proof of Theorem \ref{th} in the case $n=3$.

\section{Proof of Theorem \ref{th2} and \ref{hedg}}

 Theorem \ref{hedg} and Theorem \ref{th2} (under the additional hypothesis that the support functions of $K$ and $L$ are twice differentiable) are the direct consequences of Theorem \ref{th}.

Let $H \subset \mathbb{E}^n$ be a classical hedgehog with support function $h = h_{H} $ defined on $S^{n-1}$. Let $\alpha_k$ denote a $k$-dimensional plane passing through the origin and $\alpha_k^{\perp}$ be its orthogonal complement in $\mathbb{E}^n$. Then if $H|_{\alpha_k}$ is the projection of $H$ on $\alpha_k$ we have
$$
h_{H|_{\alpha_k}}(u) = h_{H}(u) \quad \textrm{for any} \quad u \in \alpha_k \cap S^{n-1}.
$$

This is due to the fact that for any $x \in \mathbb{E}^n$ and $u \in \alpha_k\cap S^{n-1}$ we have $h_{H}(u) = \max\{ x \cdot u : x \in H\} = \max \{(x|_{\alpha_k} + x|_{\alpha_k^{\perp}}) \cdot u : x \in H\} = \max \{x|_{\alpha_k} \cdot u: x \in H\} =h_{H|_{\alpha_k}}(u)$.

Since the requirement on the convexity can be weakened, the following two properties (see \cite{Ga}, p. 18) of support functions of convex bodies hold true for classical hedgehogs in $\mathbb{E}^n$.
\begin{enumerate}
\item For any $\phi \in O(n)$ we have $h_{\phi(H)}(u) = \max \{ x \cdot u: x \in \phi(H)\} = \max \{ \phi(z) \cdot u : z \in H \} = \max \{z \cdot \phi^T(u): z \in H \} = h_H(\phi^T(u))$;

\item For any $ a \in \mathbb{E}^n$ we have $h_{H + a}(u) = \max \{ x \cdot u : x \in H + a \} = \max \{ x \cdot u : x \in H \} + a \cdot u = h_{H}(u) + a \cdot u$.
\end{enumerate}

To conclude the proof of Theorems \ref{th2} and \ref{hedg}, we observe that the conditions on projections in the theorem can be re-written as
$$
h_M( \psi_{\xi}^T(u) ) + a_{\xi} \cdot u = h_N(u) \quad \textrm{for any} \quad u \in E(\xi),
$$
where $\psi_{\xi}$ is a rotation on $\xi^{\perp}$ and $a_{\xi} \in \xi^{\perp}$; $M$ and $ N$ are a pair of hedgehogs in the case of Theorem \ref{hedg} (or a pair of two convex bodies in the case of Theorem \ref{th2}) with the support functions $h_M$ and $h_N$ respectively. By taking $f = h_M, g = h_N $ and $\phi_{\xi} = \psi_{\xi}^T$, we conclude that $h_N(u) = h_M(u) + b \cdot u$ or $h_N(u) = h_M(-u) + b \cdot u$ for some $b \in \mathbb{E}^n$. In the first case $N = M + b$, and in the second, $N = -M+ b$.

\end{document}